\documentclass[reqno]{amsart}
\usepackage{amsmath, pb-diagram}
\usepackage{amssymb}
\usepackage{amscd}
\usepackage[all]{xy}

\renewcommand\baselinestretch{1.33}

\begin{document}
\title [Harmanci Injectivity of Modules]{Harmanci Injectivity of Modules}

\author{Burcu Ungor}
\address{Burcu Ungor, Department of Mathematics, Ankara University, 06100, Ankara, Turkey}
\email{bungor@science.ankara.edu.tr}



\date{}

\newtheorem {thm}{Theorem}[section]
\newtheorem{lem}[thm]{Lemma}
\newtheorem{prop}[thm]{Proposition}
\newtheorem{cor}[thm]{Corollary}
\newtheorem{df}[thm]{Definition}
\newtheorem{nota}{Notation}
\newtheorem{note}[thm]{Remark}
\newtheorem{ex}[thm]{Example}
\newtheorem{exs}[thm]{Examples}
\newtheorem{rem}[thm]{Remark}
\newtheorem{quo}[thm]{Question}
\newcommand{\ov}{\overline{Z}}
\renewcommand\baselinestretch{1.3}
\begin{abstract} In this paper, we are interested in a class of modules partaking in the hierarchy of injective and cotorsion modules,
so-called Harmanci injective modules, which turn out by the
motivation of relations among the concepts of injectivity,
flatness and cotorsionness. We give some characterizations and
properties of this class of modules. It is shown that the class of
all Harmanci injective modules is enveloping, and forms a perfect
cotorsion theory with the class of modules whose character modules
are Matlis injective. One of the main objectives we pursue is to
know when the injective envelope of a ring as a module over itself
is a flat module.

 \vspace{2mm}


\end{abstract}

\maketitle
\section{ Introduction }
Throughout this paper $R$ denotes an associative ring with
identity and modules are unitary $R$-modules.  The notion of
cotorsion abelian groups introduced by Harrison in \cite{Harison},
that is, an abelian group $G$ is called {\it cotorsion} if
Ext$^1_{\Bbb Z}(\Bbb Q, G) = 0$. This notion extended to modules
by Enochs in \cite{Enochs}, namely, a right $R$-module $M$ is said
to be  {\it{\rm(}Enochs{\rm)}  cotorsion} if Ext$^1_R(F, M) = 0$
for every flat right $R$-modules $F$. Let $\mathcal{A}$ be a class
of right $R$-modules. Then $\mathcal{A}^\perp =\{M_R \mid
\mbox{Ext}^1_R(A,M)=0, A\in \mathcal{A}\}$ is called the {\it
right orthogonal class} of $\mathcal{A}$ and $^\perp\mathcal{A}
=\{M_R \mid \mbox{Ext}^1_R(M,A)=0, A\in \mathcal{A}\}$ is called
the {\it left orthogonal class} of $\mathcal{A}$ (see
\cite[p.29]{X}). For any  classes of right $R$-modules
$\mathcal{A}$ and $\mathcal{B}$, if
$\mathcal{A}=\hspace{0.01cm}^\perp\mathcal{B}$ and
$\mathcal{B}=\mathcal{A}^\perp$, then the pair
$(\mathcal{A},\mathcal{B})$ is called a {\it cotorsion theory}. It
is well known that $(\mathcal{F}, \mathcal{EC})$ is a cotorsion
theory where $\mathcal{F}$ and $\mathcal{EC}$ are the classes of
all flat modules and all cotorsion modules, respectively.

The connection between a flat module and its character module was
observed by Lambek. He proved in \cite{Lambek} that a left
$R$-module $M$ is flat if and only if its character module
Hom$_{\Bbb Z}(M,\Bbb Q/\Bbb Z)$ is an injective right $R$-module.
Thus the class of all cotorsion modules is the right orthogonal
class of modules whose character modules are injective. Injective
modules, cotorsion modules and various generalizations of these
modules have been investigated in the literature by many authors.
In \cite{ES}, a right $R$-module $M$ is called {\it Whitehead} if
Ext$^1_R(M,R)=0$. As a dual notion of Whitehead modules, Yan
defined Matlis injective modules in \cite{Ya}, namely, a right
$R$-module $M$ is said to be {\it Matlis injective} if Ext$^1_
R(E(R_R),M) = 0$ where $E(R_R)$ denotes the injective envelope of
the ring $R$ as a right $R$-module. Motivated by the studies on
the modules which belong to the right orthogonal class of flat
modules, i.e., cotorsion modules, the goal of this paper is to
provide an initial contribution to the study of the right
orthogonal class of modules whose character modules are Matlis
injective. We study the behavior of modules that belong to this
right orthogonal class, so-called Harmanci injective modules. We
observe that the pair consisting of all modules whose character
modules are Matlis injective and all Harmanci injective modules is
a cotorsion theory. We are interested in the hierarchy of
injective modules and cotorsion modules, in this direction, we
show that the class of Harmanci injective modules lies strictly
between the classes of injective modules and cotorsion modules.
For a commutative Noetherian ring $R$, $E(R)$ being flat is
characterized in \cite[Theorem 5.1.3]{X}. A natural question
arises: When is $E(_RR)$ a flat left $R$-module for any ring $R$?
One of our main concerns is this question. We give an answer, that
is, the notions of injectivity and Harmanci injectivity coincide
if and only if $E(_RR)$ is flat. It is also known that Hom$_{\Bbb
Z}(M, \Bbb Q/\Bbb Z)$ is always pure injective and so cotorsion
for any module $M$ (see \cite[p.39]{X}). As an application, we
deal with character modules and also approximations of modules in
terms of Harmanci injectivity.

Let $\mathcal{C}$ be a class of right $R$-modules and $M$ a right
$R$-module. Following \cite{Eno}, a homomorphism $f\colon
C\rightarrow M$ with $C\in \mathcal{C}$ is said to be a {\it
$\mathcal{C}$-precover} of $M$ if for any homomorphism $g\colon
C'\rightarrow M$ with $C'\in \mathcal{C}$, there exists a
homomorphism $h\colon C'\rightarrow C$ such that $fh=g$. The
$\mathcal{C}$-precover $f$ is  called a {\it $\mathcal{C}$-cover}
of $M$ if any endomorphism $\alpha\colon C\rightarrow C$ with
$f\alpha = f$ is an isomorphism. The concepts of a
$\mathcal{C}$-preenvelope and $\mathcal{C}$-envelope are defined
dually. Bican, Bashir and Enochs proved the existence of a flat
cover and a cotorsion envelope for any module in \cite{BBE}. It is
also well known that every module has an injective envelope. Thus
the following question seems natural to wonder: what can we say
about the existence of Harmanci injective envelopes and covers
with respect to the class of modules whose character modules are
Matlis injective and their unique mapping properties? In
\cite{Ding}, a $\mathcal{C}$-envelope $f\colon M\rightarrow C$
with $C\in \mathcal{C}$ of a module $M$ said to have the {\it
unique mapping property} if for any homomorphism $g\colon
M\rightarrow C'$ with $C'\in \mathcal{C}$, there exists a unique
$h\colon C\rightarrow C'$ such that $hf=g$.

Briefly, we devote the first part of this paper to study another
class of modules, so-called Harmanci injective modules, which turn
out in the light of relations among the concepts of injectivity,
flatness and cotorsionness by addressing several aforementioned
questions. We devote the second part of this paper to investigate
the existence of Harmanci injective envelopes, its unique mapping
property and a cotorsion theory arising from Harmanci injectivity.

In what follows, $\Bbb Q$, $\Bbb Z$, $\Bbb Z/n\Bbb Z$ and $E(M)$
denote the ring of rational numbers, the ring of integers, the
$\Bbb Z$-module of integers modulo $n$ for a positive integer $n$
and the injective envelope of a module $M$, respectively. Also,
id$(M)$, pd$(M)$ and fd$(M)$ stand for the injective dimension,
projective dimension and flat dimension, respectively.

\section{Harmanci Injective Modules}

In this section, we consider the right orthogonal class of modules
whose character modules are Matlis injective, and call any module
in this class as Harmanci injective. Let us start the following
example for a ring $R$ whose its injective envelope $E(_RR)$ is
not a flat left $R$-module.

\begin{ex}\label{inj hull-ex}{\rm  Let $D$ be a division ring, $n$ be a positive integer,
$U_n(D)$ be the ring of $n\times n$ upper triangular matrices over
$D$ and $M_n(D)$ be the left $U_n(D)$-module of $n\times n$ full
matrices over $D$. Then the injective envelope of $U_n(D)$,
considered as a left $U_n(D)$-module, is $M_n(D)$ which is not
flat.}
\end{ex}

\begin{proof} The left $U_n(D)$-module $M_n(D)$ being an injective
envelope of $U_n(D)$ is proved in \cite{BBB}. Next we claim that
$M_n(D)$ is not a flat left $U_n(D)$-module. We prove it by
contradiction. Assume that $M_n(D)$ is a flat left
$U_n(D)$-module. Note that $M_n(D)$ is finitely generated over the
Noetherian  ring $U_n(D)$, so is finitely presented by
\cite[Corollary 3.19]{R} and, then, it is also projective by
\cite[Theorem 3.56]{R}. Hence it is a direct summand of a direct
sum of copies of $U_n(D)$, say $\bigoplus\limits_{\mathcal{I}}
U_n(D)=M_n(D)\oplus K$ for some submodule $K$ of
$\bigoplus\limits_{\mathcal{I}} U_n(D)$ where $\mathcal{I}$ is an
index set. Also, $\bigoplus\limits_{\mathcal{I}}
U_n(D)=U_n(D)\oplus L$ for some submodule $L$ of
$\bigoplus\limits_{\mathcal{I}} U_n(D)$.  Since $U_n(D)$ is a
submodule of $M_n(D)$, the modularity condition entails that
$U_n(D)$ is a direct summand of $M_n(D)$. So there exists the
natural epimorphism $\pi\colon M_n(D)\rightarrow U_n(D)$. Then
$M_n(D)/$Ker$\pi\cong U_n(D)$. Since $M_n(D)$ is an injective left
$U_n(D)$-module and $U_n(D)$ is left hereditary by \cite[Example
2.8.13]{Ro}, $M_n(D)/$Ker$\pi$ is also injective, so $U_n(D)$ is
left self-injective, but this is a contradiction, because
the homomorphism $g\colon U_n(D)\left[%
\begin{array}{cc}
  0 & 1 \\
  0 & 0 \\
\end{array}%
\right] \rightarrow U_n(D)$ defined by $g\left[%
\begin{array}{cc}
  0 & x \\
  0 & 0 \\
\end{array}%
\right]=\left[%
\begin{array}{cc}
  x & 0 \\
  0 & 0 \\
\end{array}%
\right]$ can not be extended to $U_n(D)$. Therefore  $M_n(D)$ is
not a flat left $U_n(D)$-module.
\end{proof}

By \cite[Proposition 8.18]{R}, a left $R$-module $M$ is flat if
and only if Tor$^R_1(N,M)=0$ for every right $R$-module $N$. By
means of Example \ref{inj hull-ex}, the left $R$-module $E(_RR)$
need not be flat in general. In the following, we characterize
right $R$-modules which satisfy Tor$^R_1(N, E(_RR)) = 0$.

\begin{lem}\label{equiv}  Let $N$ be a right $R$-module. Then
Tor$^R_1(N, E(_RR)) = 0$ if and only if Hom$_{\Bbb Z}(N,\Bbb
Q/\Bbb Z)$ is Matlis injective.
\end{lem}
\begin{proof}  According to \cite[Theorem 3.2.1]{Enochs-Jenda}, \begin{center} Hom$_{\Bbb
Z}($Tor$^R_1(N, E(_RR)),\Bbb Q/\Bbb Z)\cong $ Ext$^1_R(E(_RR),$
Hom$_{\Bbb Z}(N,\Bbb Q/\Bbb Z))$.\end{center} This implies
Ext$^1_R(E(_RR),$ Hom$_{\Bbb Z}(N,\Bbb Q/\Bbb Z))=0$ if and only
if Tor$^R_1(N, E(_RR))=0$.
\end{proof}

We now give our main definition, namely, Harmanci injective
modules.

\begin{df} {\rm A right $R$-module $M$ is said to be {\it Harmanci
injective} if \linebreak Tor$^R_1\big(N, E(_RR)\big) = 0$ implies
Ext$^1_R(N,M) = 0$ for every right $R$-module $N$.}
\end{df}

So clearly we have the next result.

\begin{prop}\label{inj hull} Let $R$ be a ring. Then  Hom$_{\Bbb Z}(E(_RR), \Bbb
Q/\Bbb Z)$ is a Harmanci injective right $R$-module.
\end{prop}

Obviously, every injective module is Harmanci injective.  We now
observe when the converse of this statement holds. For a
commutative Noetherian ring $R$,  $E(R)$ being flat is
characterized in \cite[Theorem 5.1.3]{X}. Via the next theorem, we
characterize $E(_RR)$ being flat for an arbitrary ring $R$. On the
other hand, this result and Example \ref{inj hull-ex} make sure
that there exists a Harmanci injective right $U_n(D)$-module which
is not injective.

\begin{thm}\label{flat} Let $R$ be a ring. Then  every Harmanci injective right
$R$-module is injective if and only if  $E(_RR)$ is flat.
\end{thm}
\begin{proof} Assume that Harmanci injectivity implies
injectivity. By Proposition \ref{inj hull}, Hom$_{\Bbb Z}(E(_RR),
\Bbb Q/\Bbb Z)$ is Harmanci injective, and so is injective. This
yields that $E(_RR)$ is flat. Suppose now that $E(_RR)$ is flat
and $M$ is a Harmanci injective right $R$-module. Let $N$ be any
right $R$-module. The module $E(_RR)$ being flat implies
Tor$^R_1(N,E(_RR))=0$. Since $M$ is Harmanci injective,
Ext$^1_R(N,M) = 0$. Therefore $M$ is injective.
\end{proof}

We immediately get the next consequences from Theorem \ref{flat}.

\begin{cor}\label{inj-har} The following hold.
\begin{enumerate}
\item Injectivity and Harmanci injectivity coincide for the
modules over a commutative  domain.
\item If $R$ is a left self-injective ring or a von Neumann regular ring, then a right
$R$-module $M$ is injective if and only if $M$ is Harmanci
injective.
\end{enumerate}
\end{cor}
\begin{proof} (1) Let $R$ be a commutative domain. Then $E(_RR)$  is the field of fractions of $R$
and it is a flat $R$-module by \cite[Corollary 5.35(i)]{R}. Hence
Theorem \ref{flat} completes the proof.\\
(2) If $R$ is a left self-injective ring, then the assertion is
obtained immediately from Theorem \ref{flat}.  If $R$ is a von
Neumann regular ring, then the proof is clear from the fact that
every module over $R$ is flat and Theorem \ref{flat}.
\end{proof}

The converse statement of (2) in Corollary \ref{inj-har}  need not
be true in general as shown below.
\begin{ex}{\rm By Corollary \ref{inj-har}(1), every Harmanci injective
$\Bbb Z$-module is injective but $\Bbb Z$ is neither
self-injective nor von Neumann regular.}
\end{ex}

In the next result, we are interested in the flat dimension of
$E(_RR)$ for any ring $R$ in terms of Harmanci injectivity.

\begin{thm} Let $R$ be a ring. Then the flat dimension of $E(_RR)$
is exactly $1$ if and only if the injective dimension of every
Harmanci injective right $R$-module is exactly $1$.
\end{thm}
\begin{proof} Firstly, we assert that fd$(E(_RR))$ is at most $1$ if and
only if id$(M)$ is at most $1$ for every Harmanci injective right
$R$-module $M$. In an attempt to prove the necessity of this
assertion, let $M$ and $N$ be right $R$-modules with $M$ Harmanci
injective. There exists an exact sequence $0\rightarrow
K\rightarrow F\rightarrow N \rightarrow 0$ where $F$ is a free
right $R$-module. Applying the functor $-\otimes_R E(_RR)$ to the
sequence, we get the exactness of

$\cdots \rightarrow $ Tor$^R_2(N,E(_RR))\rightarrow $
Tor$^R_1(K,E(_RR))\rightarrow$
Tor$^R_1(F,E(_RR))\rightarrow\cdots$. \\
Being fd$(E(_RR))\leq 1$ and flatness of $F$ imply that
Tor$^R_1(K,E(_RR))=0$. Since $M$ is Harmanci injective,
Ext$^1_R(K,M)=0$. On the other hand, we also have the exact
sequence $0=$ Ext$^1_R(K,M)\rightarrow$ Ext$^2_R(N,M)\rightarrow$
Ext$^2_R(F,M)\rightarrow\cdots$. As $F$ is projective,
Ext$^2_R(F,M)=0$. Thus Ext$^2_R(N,M)=0$. This implies id$(M)\leq
1$. For the sufficiency, let $N$ be a right $R$-module. By
Proposition \ref{inj hull} and hypothesis, id$($Hom$_{\Bbb
Z}(E(_RR), \Bbb Q/\Bbb Z))\leq 1$. This yields Ext$^2_R(N,$
Hom$_{\Bbb Z}(E(_RR), \Bbb Q/\Bbb Z))=0$. So by the isomorphism
Ext$^2_R(N,$ Hom$_{\Bbb Z}(E(_RR), \Bbb Q/\Bbb Z))\cong$
Hom$_{\Bbb Z}($Tor$^R_2(N, E(_RR)),\Bbb Q/\Bbb Z)$, we have
Tor$^R_2(N,E(_RR))=0$. Therefore fd$(E(_RR))\leq 1$. Now we
complete the proof in the light of this authenticated assertion
and Theorem \ref{flat}.
\end{proof}

The next result shows that the class of Harmanci injective modules
lies between those of injective modules and cotorsion modules.

\begin{prop}\label{cotor} Every Harmanci injective right $R$-module is cotorsion.
\end{prop}
\begin{proof} Let $M$ and $F$ be right $R$-modules with $M$ Harmanci injective and $F$  flat.
Then Tor$^R_1(F,E(_RR))=0$. Hence Harmanci injectivity of $M$
implies Ext$^1_R(F,M)=0$. Thus $M$ is cotorsion.
\end{proof}

The next examples show that the converse of Proposition
\ref{cotor} need not be hold in general.

\begin{ex} {\rm(1) The $\Bbb Z$-module $M=$ Hom$_{\Bbb Z}(\Bbb Z/2\Bbb Z, \Bbb Q/\Bbb
Z)$ is pure-injective as it is the character module of $\Bbb
Z/2\Bbb Z$ and so it is cotorsion. On the other hand, $M$ is not
injective because $\Bbb Z/2\Bbb Z$ is not flat. Then Corollary
\ref{inj-har}(1) implies that $M$ is not
Harmanci injective.\\
(2) Let $R$ be a quasi-Frobenius (shortly, QF) ring which is not
right pure-semisimple. Then $R$ is right perfect. Since every flat
right $R$-module is projective, every right $R$-module is
cotorsion. On the other hand, there is a right $R$-module $M$
which is not pure-injective as $R$ is not right pure-semisimple.
Hence $M$ is not injective. The ring $R$ being left self-injective
implies $M$ is not Harmanci injective by Corollary
\ref{inj-har}(2).}
\end{ex}

Recall that a right $R$-module $M$ is said to be {\it divisible}
if Ext$^1_R(R/aR, M) = 0$ for all $a\in R$. By Corollary
\ref{inj-har}(1) and \cite[Corollary 3.35(i)]{R}, if $R$ is a
principal ideal domain, then an $R$-module $M$ is Harmanci
injective if and only if $M$ is injective if and only if $M$ is
divisible. In the next result, we investigate when Harmanci
injectivity implies being divisible.

\begin{prop} Let $R$ be a ring. If every principal right ideal of
$R$ is pure, then every Harmanci injective right $R$-module is
divisible.
\end{prop}
\begin{proof} Let $M$ be a Harmanci injective right $R$-module and
$a\in R$. Consider the short exact sequence $0\rightarrow
aR\rightarrow R \rightarrow R/aR\rightarrow 0$. Then we have the
exactness of $0=$ Tor$_1^R(R,E(_RR))\rightarrow$
Tor$_1^R(R/aR,E(_RR))\rightarrow aR\otimes E(_RR)\rightarrow R
\otimes E(_RR)\rightarrow R/aR\otimes E(_RR)\rightarrow 0$. Since
$aR$ is pure in $R$, the homomorphism $aR\otimes E(_RR)\rightarrow
R \otimes E(_RR)$ is monic. It follows that
Tor$_1^R(R/aR,E(_RR))=0$. The module $M$ being Harmanci injective
implies Ext$^1_R(R/aR, M)=0$. Therefore $M$ is divisible.
\end{proof}

For any ring $R$, owing to Proposition \ref{inj hull}, the
character module of $E(_RR)$ is always Harmanci injective. In the
next result, we investigate  Harmanci injectivity of a character
module of any module.
\begin{thm}\label{char} Let $S$ be a ring and $N$ a left $S$-module. Then the following are equivalent.
\begin{enumerate}
    \item[(1)] Hom$_{\Bbb Z}(N, \Bbb Q/\Bbb Z)$ is a Harmanci
    injective right $S$-module.
    \item[(2)] For every ring $R$ with $N$ a left $S$-right
    $R$-bimodule and every injective right $R$-module $M$, Hom$_R(N,M)$ is  Harmanci injective as a right
    $S$-module.
\end{enumerate}
\end{thm}
\begin{proof} (1) $\Rightarrow$ (2) Let $N$ be a left $S$-right
$R$-bimodule, $M$ an injective right $R$-module and $K$ a right
$S$-module with Tor$^S_1\big(K, E(_SS)\big) = 0$. We claim that
Ext$^1_S(K, $ Hom$_R(N,M))=0$. Harmanci injectivity of the right
$S$-module Hom$_{\Bbb Z}(N, \Bbb Q/\Bbb Z)$ yields Ext$^1_S(K, $
Hom$_{\Bbb Z}(N, \Bbb Q/\Bbb Z))=0$. It follows Hom$_{\Bbb
Z}($Tor$_1^S(K,N),\Bbb Q/\Bbb Z)=0$, and so Tor$_1^S(K,N)=0$.
Hence Hom$_R($Tor$_1^S(K,N),M)=0$. Thus injectivity of $M$ implies
Ext$^1_S(K, $ Hom$_R(N,M))=0$. Therefore the right $S$-module
Hom$_R(N,M)$ is  Harmanci injective.\\
(2) $\Rightarrow$ (1) Obvious by taking $R=\Bbb Z$ and $M=\Bbb
Q/\Bbb Z$.
\end{proof}

Owing to Theorem \ref{char} and Proposition \ref{inj hull}, we
acquire the next result.
\begin{cor} Let $S$ be a ring. Then for every ring $R$ with $E(_SS)$ a left $S$-right
    $R$-bimodule and every injective right $R$-module $M$, Hom$_R(E(_SS),M)$ is a Harmanci injective right
    $S$-module.
\end{cor}

In the sequel, let  $\mathcal{HI}$ denote the class of all
Harmanci injective right $R$-modules and $\mathcal{CMI}$ stand for
the class of all right $R$-modules whose character modules are
Matlis injective. In the following, we mention some properties of
the class $\mathcal{HI}$.

\begin{lem}\label{basic} Harmanci injective modules satisfy the following properties.
\begin{enumerate}
\item[(1)] Any direct product of Harmanci injective modules is
Harmanci injective.
\item[(2)] A finite direct sum of Harmanci injective modules is
Harmanci injective.
\item[(3)] Direct summands of Harmanci injective modules are
Harmanci injective.
\item[(4)] The class of Harmanci injective modules is closed
under extensions.
\end{enumerate}
\end{lem}
\begin{proof} (1) Let $\{M_i\}_{i\in I}$ be a collection of Harmanci injective right $R$-modules, $M = \prod\limits_{i\in I}M_i$
and  $N$ be a right $R$-module with Tor$^R_1(N, E(_RR)) = 0$. Then
Ext$^1_R(N,M_i) = 0$ for each $i\in I$. Since Ext$^1_R(N,M) \cong
\prod\limits_{i\in I}$ Ext$^1_R(N,M_i)$, we have Ext$^1_R(N,M) = 0$.\\
(2) It is clear by (1).\\(3) Let $M = M_1\oplus M_2$ be a Harmanci
injective right $R$-module. Let Tor$^R_1(N, E(_RR)) = 0$ for some
right $R$-module $N$. By hypothesis, Ext$^1_R(N,M) = 0$. Since
Ext$^1_R(N, M) \cong$ Ext$^1_R(N, M_1)\oplus$ Ext$^1_R(N, M_2)$,
Ext$^1_R(N,M_1) =$ Ext$^1_R(N,M_2)= 0$.\\
(4) Let $0\rightarrow M \rightarrow N\rightarrow K\rightarrow 0$
be an exact sequence of right $R$-modules with $M$ and $K$
Harmanci injective. Let $L$ be a right $R$-module with Tor$^R_1(L,
E(_RR)) = 0$. Then we obtain $\cdots \rightarrow $
Ext$_R^1(L,M)\rightarrow$ Ext$_R^1(L,N)\rightarrow$
Ext$_R^1(L,K)\rightarrow \cdots$. By hypothesis, Ext$_R^1(L,M)=$
Ext$_R^1(L,K)=0$. This yields Ext$_R^1(L,N)=0$.
\end{proof}

\begin{prop} Let $M$ be a Harmanci injective right $R$-module and $K$ a
submodule of $M$ with injective dimension at most $1$. Then $M/K$
is Harmanci injective.
\end{prop}
\begin{proof} Consider the exact sequence $0\rightarrow K\rightarrow M\rightarrow M/K \rightarrow
0$. Then we have the long exact sequence $\cdots \rightarrow$
Ext$^1_R(N,M)\rightarrow $ Ext$^1_R(N,M/K)\rightarrow$
Ext$^2_R(N,K)\rightarrow \cdots $\\ by means of the functor
Hom$_R(N,-)$ where $N$ is a right $R$-module such that
Tor$^R_1(N,E(_RR))=0$. Since $M$ is Harmanci injective,
Ext$^1_R(N,M)=0$, also Ext$^2_R(N,K)=0$ by hypothesis. Therefore
Ext$^1_R(N,M/K)=0$.
\end{proof}

We now characterize  quotients of Harmanci injective modules being
Harmanci injective. Recall that a module is said to be {\it
h-divisible} if it is an epic image of an injective module. Let
$M$ be a right $R$-module and $K$ a submodule of $M$. We call $K$
an {\it $E(_RR)$-pure submodule} of $M$ if the sequence
$0\rightarrow K\otimes_R E(_RR)\rightarrow M\otimes_R E(_RR)$ is
exact. Obviously, every pure submodule of a right $R$-module is
$E(_RR)$-pure.

\begin{prop} The following are equivalent.
\begin{enumerate}
    \item[(1)] The class $\mathcal{HI}$  is closed
    under homomorphic images.
    \item[(2)] Every h-divisible right $R$-module is
    Harmanci injective.
    \item[(3)] Every module which belongs to $\mathcal{CMI}$  has projective dimension at most $1$.
    \item[(4)] Every $E(_RR)$-pure submodule of projective modules
    is projective.
\end{enumerate}
In this case, the projective dimension of a flat right $R$-module
is at most $1$, equivalently, pure submodules of projective
modules are also projective.
\end{prop}
\begin{proof} (1) $\Rightarrow$ (2) Obvious.\\
(2) $\Rightarrow$ (3) Let $N\in \mathcal{CMI}$ and $M$ any right
$R$-module. We claim that Ext$_R^2(N,M)=0$. Applying the functor
Hom$_R(N,-)$ to the exact sequence $0\rightarrow M \rightarrow
E(M)\rightarrow E(M)/M\rightarrow 0$, we have the exactness of

$\cdots \rightarrow $ Ext$_R^1(N, E(M)/M)\rightarrow$
Ext$_R^2(N,M)\rightarrow$ Ext$_R^2(N,E(M))\rightarrow \cdots$.
\\By (2), Ext$_R^1(N, E(M)/M)=0$, also the injectivity of $E(M)$
implies Ext$_R^2(N,E(M))=0$. Thus Ext$_R^2(N,M)=0$, as desired.\\
(3) $\Rightarrow$ (1) Let $M$ be a Harmanci injective right
$R$-module, $K$ a submodule of $M$ and $N$ a right $R$-module with
Tor$^R_1(N, E(_RR)) = 0$. Consider the exact sequence
$0\rightarrow K \rightarrow M\rightarrow M/K\rightarrow 0$. The
functor Hom$_R(N,-)$ yields the exact sequence

$\cdots \rightarrow $ Ext$_R^1(N, M)\rightarrow$
Ext$_R^1(N,M/K)\rightarrow$ Ext$_R^2(N,K)\rightarrow \cdots$.\\
Harmanci injectivity of $M$ and (3) imply Ext$_R^1(N, M)=$
Ext$_R^2(N,K)=0$. Therefore Ext$_R^1(N,M/K)=0$. \\
(3) $\Rightarrow$ (4) Let $P$ be a projective right $R$-module and
$K$ be an $E(_RR)$-pure submodule of $P$. Applying the functor
$-\otimes_R E(_RR)$ to the short exact sequence $0\rightarrow K
\rightarrow P \rightarrow P/K \rightarrow0$, we have  $\cdots
\rightarrow$ Tor$^R_1(P,E(_RR))\rightarrow $ Tor$^R_1(P/K,
E(_RR))\rightarrow K\otimes_R E(_RR)\rightarrow P\otimes_R
E(_RR)\rightarrow \cdots$. The module $P$ being projective and $K$
being an $E(_RR)$-pure submodule of $P$ imply Tor$^R_1(P/K,
E(_RR))=0$. Then pd$(P/K)\leq 1$ by (3). If pd$(K)>0=$ pd$(P)$,
then pd$(P/K)\geq 2$ by \cite[p.466, Ex.8.5(ii)]{R}. This
contradiction yields pd$(K)=0$, i.e, $K$ is
projective.\\
(4) $\Rightarrow$ (3) Let $M\in \mathcal{CMI}$. There exists a
short exact sequence $0\rightarrow K \rightarrow F \rightarrow M
\rightarrow 0$ where $F$ is free. Since Tor$^R_1(M, E(_RR))=0$,
$K$ is an $E(_RR)$-pure submodule of $F$. By (4), $K$ is
projective. In the light of \cite[p.466, Ex.8.5(iii)]{R},
pd$(M)\leq 1$.
\end{proof}

Recall that a class $\mathcal{C}$ is called {\it coresolving}
provided that $\mathcal{C}$ is closed under extensions, every
injective module is in $\mathcal{C}$ and $C \in \mathcal{C}$
whenever $0 \rightarrow A \rightarrow B \rightarrow C \rightarrow
0$ is a short exact sequence such that $A,B \in \mathcal{C}$. In
the light of  the fact that injectivity implies Harmanci
injectivity and Lemma \ref{basic}(4), we now address the following
question: When is the class $\mathcal{HI}$ coresolving?

\begin{prop} The following are equivalent.
\begin{enumerate}
\item[(1)] For every exact sequence $0\rightarrow M \rightarrow N\rightarrow K\rightarrow
0$ of right $R$-modules, if $M$ and $N$ are Harmanci injective,
then $K$ is Harmanci injective.
\item[(2)] If $M$ is a Harmanci injective right $R$-module,
then $E(M)/M$ is Harmanci injective.
\item[(3)] If $M$ is a Harmanci injective right $R$-module,
then for any right $R$-module $N$, being Tor$^R_1(N, E(_RR)) = 0$
implies Ext$^n_R(N,M) = 0$  where $n\geq 2$.
\end{enumerate}
In this case, $\mathcal{HI}$ is a coresolving class.
\end{prop}
\begin{proof} (1) $\Rightarrow$ (2) Clear.\\
(2) $\Rightarrow$ (3) Let $M$ and $N$ be right $R$-modules with
$M$  Harmanci injective and  Tor$^R_1\big(N, E(_RR)\big) = 0$. If
we apply the functor Hom$_R(N,-)$ to the exact sequence
$0\rightarrow M \rightarrow E(M)\rightarrow E(M)/M\rightarrow 0$,
then we obtain

$\cdots \rightarrow $ Ext$_R^1(N, E(M)/M)\rightarrow$
Ext$_R^2(N,M)\rightarrow$ Ext$_R^2(N,E(M))\rightarrow \cdots$.\\
By (2), Ext$_R^1(N, E(M)/M)= 0$ and by the injectivity of $E(M)$,
Ext$_R^2(N,E(M))=0$. Hence Ext$_R^2(N,M)=0$. Since $E(M)/M$ is
Harmanci injective, by the similar discussion above, we have
Ext$_R^2(N,E(M)/M)=0$. This yields Ext$_R^3(N,M)=0$. Continuing in
this way, by induction on $n\geq 2$, Ext$^n_R(N,M) = 0$ is
obtained.\\
(3) $\Rightarrow$ (1) Consider an exact sequence $0\rightarrow M
\rightarrow N\rightarrow K\rightarrow 0$ of right $R$-modules with
$M$ and $N$ Harmanci injective. Let $L$ be a right $R$-module with
Tor$^R_1(L, E(_RR)) = 0$. Then we have the exact sequence $\cdots
\rightarrow $ Ext$_R^1(L,N)\rightarrow$ Ext$_R^1(L,K)\rightarrow$
Ext$_R^2(L,M)\rightarrow \cdots$. By hypothesis, Ext$_R^1(L,N)=$
Ext$_R^2(L,M)=0$. It follows that Ext$_R^1(L,K)=0$, establishing
the result.
\end{proof}

We now investigate when the character module of a Harmanci
injective module is Matlis injective. We need the next lemma for
this investigation. By this means, some properties of the class
$\mathcal{CMI}$ are acquired.

\begin{lem}\label{C-class} The following hold.
\begin{enumerate}
    \item[(1)] $\mathcal{CMI}$ is closed under extensions.
    \item[(2)] $\mathcal{CMI}$ is closed under direct summands.
    \item[(3)] $\mathcal{CMI}$ is closed under direct sums.
    \item[(4)] $\mathcal{CMI}$ is closed under pure quotients.
    \item[(5)] $\mathcal{CMI}$ is a covering class.
    \item[(6)] The kernel of every $\mathcal{CMI}$-cover is  Harmanci injective.
    \item[(7)] Every right $R$-module has a $\mathcal{CMI}$-cover
    with the unique mapping property if and only if for every exact sequence
    $A\rightarrow B \rightarrow C\rightarrow 0$ of right
    $R$-modules, being $A, B\in \mathcal{CMI}$ implies $C\in
    \mathcal{CMI}$.
\end{enumerate}
\end{lem}
\begin{proof} Note that by Lemma \ref{equiv}, $\mathcal{CMI}=\{N_R \mid $ Hom$_{\Bbb Z}(N, \Bbb Q/\Bbb Z)$ is Matlis injective$\}=
\{N_R \mid $ Tor$^R_1(N, E(_RR)) = 0\}$.\\
(1) Let $0\rightarrow A\rightarrow B \rightarrow C\rightarrow 0$
be an exact sequence of right $R$-modules with $A, C\in
\mathcal{CMI}$. We claim that $B\in \mathcal{CMI}$. If we apply
the functor $-\otimes_R E(_RR)$ to the sequence, we get the long
exact sequence
$$\cdots \rightarrow \mbox{Tor}^R_1(A, E(_RR)) \rightarrow
\mbox{Tor}^R_1(B, E(_RR)) \rightarrow \mbox{Tor}^R_1(C, E(_RR))
\rightarrow A \otimes_R E(_RR)\rightarrow \cdots.$$ Since $A, C\in
\mathcal{CMI}$, Tor$^R_1(A, E(_RR))=$ Tor$^R_1(C, E(_RR))=0$. This
implies that Tor$^R_1(B, E(_RR))=0$. Hence $B\in \mathcal{CMI}$,
as
desired.\\
(2) Let $N\in \mathcal{CMI}$ and assume that $N$ has a
decomposition $N=K\oplus L$. Then $0=$ Tor$^R_1(N, E(_RR))\cong $
Tor$^R_1(K, E(_RR))\oplus$ Tor$^R_1(L, E(_RR))$. Hence Tor$^R_1(K,
E(_RR))=$
Tor$^R_1(L, E(_RR))=0$. Hence $K, L\in \mathcal{CMI}$.\\
(3) Let $\{M_i\}_{i\in I}$ be a family of modules for an index set
$I$ with $M_i\in \mathcal{CMI}$ for each $i\in I$. It is known
that Tor$^R_1(\bigoplus \limits_{i\in I} M_i, E(_RR))\cong
\bigoplus \limits_{i\in I}$ Tor$^R_1(M_i, E(_RR))$. For each $i\in
I$, being Tor$^R_1(M_i, E(_RR))=0$ implies Tor$^R_1(\bigoplus
\limits_{i\in I} M_i, E(_RR))=0$. Therefore $\bigoplus
\limits_{i\in I} M_i\in
\mathcal{CMI}$.\\
(4) Let $M\in \mathcal{CMI}$ and $N$ a pure submodule of $M$. We
show that $M/N$ belongs to $\mathcal{CMI}$. If we apply the
functor $-\otimes_{R}E(_RR)$ to the exact sequence $0\rightarrow
N\rightarrow M \rightarrow M/N\rightarrow 0$, the long exact
sequence $\cdots \rightarrow $ Tor$^R_1(M,E(_RR))\rightarrow $
Tor$^R_1(M/N,E(_RR))\rightarrow N\otimes_R E(_RR)\rightarrow
M\otimes_R E(_RR)\rightarrow M/N\otimes_R E(_RR)\rightarrow 0$ is
obtained. Since $M\in \mathcal{CMI}$ and $N\otimes_R
E(_RR)\rightarrow M\otimes_R E(_RR)$ is monic, Tor$^R_1(M/N,E(_RR))=0$. \\
(5) Clear by (3), (4) and \cite[Theorem 2.5]{HJ}. \\
(6) It follows from (1) and Wakamatsu's Lemma \cite[Lemma
2.1.1]{X}. \\
(7) Assume that every right $R$-module has a $\mathcal{CMI}$-cover
with the unique mapping property and consider an exact sequence
$A\overset{f}\rightarrow B \overset{g}\rightarrow C\rightarrow 0$
where $A, B\in \mathcal{CMI}$. Let $h\colon D\rightarrow C$ be a
$\mathcal{CMI}$-cover of $C$. Then there exists a homomorphism
$\alpha\colon B\rightarrow D$ such that $h\alpha = g$ as shown in
the following diagram.
\begin{center}
$\xymatrix{ &&D\ar[d]_h&\\
A \ar[r]_f &B\ar[r]_g \ar@{.>}[ur]^\alpha
&C\ar[r]\ar@{.>}@/_/[u]_\beta&0\\
 } $
\end{center}
Hence $h\alpha f=gf=0$ because of the exactness of the sequence.
The unique mapping property and being $h(\alpha f)=h0$ imply
$\alpha f=0$. Thus Ker$g=$ Im$f\subseteq $ Ker$\alpha$. So Factor
Theorem yields a homomorphism $\beta\colon C\rightarrow D$ such
that $\beta g=\alpha$. Then $h\beta g=h\alpha=g=1_Cg$, and so $h
\beta=1_C$ because $g$ is an epimorphism. It follows that $D=$
Ker$h~\oplus $ Im$\beta$ and  $\beta$ is a monomorphism. Thus
$C\cong $ Im$\beta$. Being $D\in \mathcal{CMI}$ implies $C\in
\mathcal{CMI}$ by (2). Conversely, let $M$ be a right $R$-module.
By (5), there is a $\mathcal{CMI}$-cover $f\colon F\rightarrow M$
of $M$. Suppose that for any $G\in \mathcal{CMI}$ and  $g\colon
G\rightarrow M$, there exist $h_1, h_2 \colon G\rightarrow F$ such
that $fh_1=fh_2=g$. Then $f(h_1-h_2)=0$, so Im$(h_1-h_2)\subseteq
$ Ker$f$. Hence there is a homomorphism $\alpha \colon
F/$Im$(h_1-h_2)\rightarrow M$ with $\alpha \pi=f$ by the Factor
Theorem where $\pi\colon F\rightarrow F/$Im$(h_1-h_2)$ is the
natural projection. On the other hand, the exactness of
$G\rightarrow F\rightarrow F/$Im$(h_1-h_2)\rightarrow 0$ implies
$F/$Im$(h_1-h_2)\in \mathcal{CMI}$ by hypothesis. It follows that
there exists $\beta\colon F/$Im$(h_1-h_2)\rightarrow F$ with
$f\beta=\alpha$. Thus $f(\beta\pi)=\alpha\pi=f$. Since $f$ is a
$\mathcal{CMI}$-cover, $\beta\pi$ is an isomorphism. This yields
$\pi$ is a monomorphism, and so Ker$\pi=$ Im$(h_1-h_2)=0$. This
implies that $h_1=h_2$. This completes the proof.
\end{proof}

\begin{thm}\label{hom image}  Consider the
following conditions.
\begin{enumerate}
    \item[(1)] $\mathcal{CMI}$ is closed under homomorphic images.
    \item[(2)] Every Harmanci injective right $R$-module belongs to $\mathcal{CMI}$.
\end{enumerate}
Then {\rm (1)} $\Rightarrow$ {\rm (2)}. Moreover, if every
$\mathcal{CMI}$-cover satisfies the unique mapping property, then
{\rm (2)} $\Rightarrow$ {\rm (1)}.
\end{thm}
\begin{proof} (1) $\Rightarrow$ (2) Let $M$ be a Harmanci injective right $R$-module.
Then $M$ is a homomorphic image of a flat module $F$. Since the
character module of $F$ is injective, so is Matlis injective,
$F\in \mathcal{CMI}$. Hence $M\in \mathcal{CMI}$ by (1).

Now assume that every $\mathcal{CMI}$-cover satisfies the unique
mapping property.\\
(2) $\Rightarrow$ (1) Let $M$ be a homomorphic image of a module
in $\mathcal{CMI}$. We claim that $M\in \mathcal{CMI}$. By Lemma
\ref{C-class}(5) and assumption, $M$ has a $\mathcal{CMI}$-cover
with the unique mapping property, say $f\colon F\rightarrow M$.
Then we have an exact sequence $0\rightarrow
\mbox{Ker}f\rightarrow F \overset{f}\rightarrow M\rightarrow 0$
with $F\in \mathcal{CMI}$. By Lemma \ref{C-class}(6),  Ker$f$ is
Harmanci injective. Hence (2) implies Ker$f\in \mathcal{CMI}$. By
taking into account of Lemma \ref{C-class}(7), we have $M\in
\mathcal{CMI}$.
\end{proof}

We close this section by observing some characterizations of
Harmanci injectivity.

\begin{prop}\label{exact} The following are equivalent for a right $R$-module $M$.
\begin{enumerate}
    \item[(1)] $M$ is Harmanci injective.
     \item[(2)] $M$ is injective relative to every exact sequence $0\rightarrow A\rightarrow B\rightarrow C \rightarrow
    0$ of right $R$-modules with $C\in \mathcal{CMI}$.
    \item[(3)] Every  right $R$-module $N\in  \mathcal{CMI}$  is projective relative to any exact
    sequence $0\rightarrow M\rightarrow B\rightarrow C \rightarrow
    0$ of right $R$-modules.
\end{enumerate}
\end{prop}
\begin{proof} (1) $\Rightarrow$ (2) Assume that
$0\rightarrow A\rightarrow B\rightarrow C \rightarrow 0$ is an
exact sequence where  $C\in \mathcal{CMI}$. Being Ext$^1_R(C,M)=0$
gives rise to the exactness of Hom$_R(B,M)\rightarrow$
Hom$_R(A,M)\rightarrow 0$. Therefore $M$ is injective relative to
the aforementioned sequence. \\
(2) $\Rightarrow$ (1) Let $N\in \mathcal{CMI}$ be a right
$R$-module and $0\rightarrow M \overset{f}\rightarrow X
\rightarrow N\rightarrow 0$ an exact sequence for a right
$R$-module $X$. (2) yields the exactness of
Hom$_R(X,M)\rightarrow$ Hom$_R(M,M)\rightarrow 0$. Hence there
exists a homomorphism $g\colon X\rightarrow M$ such that $gf=1_M$.
It follows that the sequence $0\rightarrow M \rightarrow X
\rightarrow N\rightarrow 0$ is split. Therefore
Ext$^1_R(N,M)=0$.\\
(1) $\Leftrightarrow$ (3) It is a dual of the proof of (1)
$\Leftrightarrow$ (2).
\end{proof}

\begin{thm}\label{Hom} Let $R$ be a commutative ring. Then the following are
equivalent for an $R$-module $M$.\begin{enumerate}
    \item[(1)] $M$ is Harmanci injective.
    \item[(2)] Hom$_R(N,M)$ is Harmanci injective for every flat
    $R$-module $N$.
    \item[(3)] Hom$_R(N,M)$ is Harmanci injective for every
    projective $R$-module $N$.
    \item[(4)] Hom$_R(N,M)$ is Harmanci injective for every free
    $R$-module $N$.
\end{enumerate}
\end{thm}

\begin{proof} (1) $\Rightarrow$ (2) Let $N$ and $K$ be $R$-modules
with $N$ flat and Tor$^R_1(K, E(_RR))=0$. There exists an exact
sequence $0\rightarrow L \rightarrow F \rightarrow K \rightarrow
0$ where $F$ is free. Since $N$ is flat,  the  sequence
$0\rightarrow N\otimes_R L\rightarrow N\otimes_R F\rightarrow
N\otimes_R K\rightarrow 0$ is exact. It follows that
\begin{center}$\cdots \rightarrow$ Hom$_R(N\otimes_R F,
M)\rightarrow$ Hom$_R(N\otimes_R L,M)\rightarrow$
Ext$^1_R(N\otimes_R K,M)\rightarrow \cdots $
\end{center} is also exact. Since $N$ is flat and Tor$^R_1(K, E(_RR))=0$,
Tor$^R_1(N\otimes_R K, E(_RR))\cong N\otimes_R$
Tor$^R_1(K,E(_RR))=0$ by \cite[p. 667]{R}. This implies that
Ext$^1_R(N\otimes_R K,M)=0$ by (1). Hence via the Adjoint
Isomorphism, we have the exactness of the sequence

Hom$_R(F,$ Hom$_R(N,M))\rightarrow$ Hom$_R(L,$
Hom$_R(N,M))\rightarrow 0$.\\ Then again, according to the short
exact sequence $0\rightarrow L \rightarrow F \rightarrow
K\rightarrow 0$ and the functor Hom$_R(-,$ Hom$_R(N,M))$, we
obtain  \begin{center}{\small $\cdots \rightarrow$ Hom$_R(F,$
Hom$_R(N,M))\rightarrow$ Hom$_R(L,$ Hom$_R(N,M))\rightarrow$
Ext$^1_R(K,$ Hom$_R(N,M))\rightarrow$ Ext$^1_R(F,$
Hom$_R(N,M))=0$.}\end{center} Therefore Ext$^1_R(K,$
Hom$_R(N,M))=0$, as desired.\\
(2) $\Rightarrow$ (3) $\Rightarrow$ (4) Obvious.\\
(4) $\Rightarrow$ (1) Clear from the fact that $M\cong $
Hom$_R(R,M)$.
\end{proof}

We illustrate some consequences of Theorem \ref{Hom} as follows.

\begin{cor} Let $R$ be a commutative ring and $M$ a flat
$R$-module. If $M$ is fully invariant in a Harmanci injective
$R$-module, then End$(M)$ is also a Harmanci injective $R$-module.
\end{cor}

\begin{cor}The following are
equivalent for a ring $R$.\begin{enumerate}
    \item[(1)] $R$ is Harmanci injective as a right $R$-module.
    \item[(2)] Hom$_R(N,R)$ is a Harmanci injective right $R$-module for every $R$-$R$ bimodule  $N$ with $N$ flat as a left $R$-module.
    \item[(3)] Hom$_R(N,R)$ is a Harmanci injective right $R$-module for every $R$-$R$ bimodule  $N$ with $N$ projective as a left $R$-module.
    \item[(4)] Hom$_R(N,R)$ is a Harmanci injective right $R$-module for every $R$-$R$ bimodule  $N$ with $N$ free as a left $R$-module.
    \item[(5)] Every direct product of the left $R$-module $R$ is
    Harmanci injective as a right $R$-module.
\end{enumerate}
\end{cor}

\section{The pair $(\mathcal{CMI}, \mathcal{HI})$}

In this section, we are particularly interested in some properties
of the pair $(\mathcal{CMI}, \mathcal{HI})$ and approximations of
modules in terms of Harmanci injectivity.

\begin{thm}\label{pair} The pair $(\mathcal{CMI}, \mathcal{HI})$ is a cotorsion theory.
\end{thm}
\begin{proof} Obviously, $\mathcal{HI}=\mathcal{CMI}^\bot$. We need
to show that $\mathcal{CMI}=$ $^\bot\mathcal{HI}$. By the
definition, it is clear that $\mathcal{CMI}\subseteq$
$^\bot\mathcal{HI}$. For the reverse inclusion, let $N\in$
$^\bot\mathcal{HI}$. By Proposition \ref{inj hull}, Ext$^1_R(N,$
Hom$_{\Bbb Z}(E(_RR), \Bbb Q/\Bbb Z))=0$. So being Ext$^1_R(N,$
Hom$_{\Bbb Z}(E(_RR), \Bbb Q/\Bbb Z))\cong $ Hom$_{\Bbb
Z}($Tor$^R_1(N, E(_RR)), \Bbb Q/\Bbb Z)$ implies Tor$^R_1(N,
E(_RR))=0$, that is Hom$_{\Bbb Z}(N,\Bbb Q/\Bbb Z)$ is Matlis
injective by Lemma \ref{equiv}. Thus $N\in \mathcal{CMI}$.
Therefore $^\bot\mathcal{HI}\subseteq \mathcal{CMI}$. This
completes the proof.
\end{proof}

\begin{rem} In the light of Theorem \ref{pair}, we have the following.
\begin{enumerate}
    \item The following are equivalent.
\begin{enumerate}
    \item[(i)] Every cotorsion right $R$-module is Harmanci injective.
    \item[(ii)] Every module which belongs to $\mathcal{CMI}$  is
flat.
\end{enumerate}
    \item  The following are equivalent.
\begin{enumerate}
    \item[(i)] Every right $R$-module is Harmanci injective.
    \item[(ii)] Every module which belongs to $\mathcal{CMI}$  is
projective.
\end{enumerate}
    \item If every right $R$-module is Harmanci injective, then
$R$ is right perfect.
\end{enumerate}
\end{rem}

Let $(\mathcal {F}, \mathcal T)$ be a cotorsion pair. It is called
{\it perfect} if $\mathcal F$ is a covering class and $\mathcal T$
is an enveloping class. Also, $(\mathcal {F}, \mathcal T)$ is {\it
complete} if each module has a special $\mathcal {F}$-precover,
equivalently, each module has a special $\mathcal {T}$-preenvelope
(see \cite[Lemma 1.17]{Tr}). Let $\mathcal C \subseteq$ Mod-$R$
and $M\in$ Mod-$R$ and $f \colon M\rightarrow C$ be a $\mathcal
C$-preenvelope of $M$. It is called {\it special} if $f$ is
injective and Ext$^1_R(\mbox{Coker} f, T) = 0$ for all $T\in
\mathcal C$. Let $f \colon C\rightarrow M$ be a $\mathcal
C$-precover. It is called {\it special} if $f$ is surjective and
Ext$^1_R(F, \mbox{Ker} f) = 0$ for all $F\in \mathcal C$. Given
the above concepts, we now address the following question: What
can be said about such properties for the cotorsion pair
$(\mathcal{CMI}, \mathcal{HI})$?

\begin{thm}\label{properties} The following hold.
\begin{enumerate}
\item[(1)]  The cotorsion theory $(\mathcal{CMI}, \mathcal{HI})$ is
complete.
\item[(2)]  $\mathcal{CMI}$ is a special covering class.
\item[(3)] $\mathcal{HI}$ is a special enveloping class.
\item[(4)] The cotorsion theory $(\mathcal{CMI}, \mathcal{HI})$ is
perfect.
\end{enumerate}
\end{thm}

\begin{proof} (1) Follows from \cite[Lemma 1.9(2) and Lemma 1.13]{T} and Lemma \ref{C-class}(1) and (5).\\
(2) By Lemma \ref{C-class}(5) each module has a
$\mathcal{CMI}$-cover. Since the class of all projective modules
is contained in $\mathcal{CMI}$, every $\mathcal{CMI}$-cover is an
epimorphism. On the other hand, for a  $\mathcal{CMI}$-cover $f$,
Lemma \ref{C-class}(6) implies that Ext$^1_R(F, \mbox{Ker} f) = 0$
for all $F\in \mathcal {CMI}$. Thus every $\mathcal{CMI}$-cover is
special.\\
(3) Since Tor commutes with direct limits (see \cite[Proposition
7.8]{R}), the class $\mathcal{CMI}$ is closed under direct limits.
Hence (1) and \cite[Corollary 1.19]{T} imply that $\mathcal{HI}$
is an enveloping class. As the class of all injective modules is
contained in $\mathcal{HI}$, every $\mathcal{HI}$-envelope is a
monomorphism. For a $\mathcal{HI}$-envelope $f$, \cite[Lemma
2.1.2]{X} and Lemma \ref{basic}(4) yield that
Ext$^1_R(\mbox{Coker} f, X) = 0$ for all $X\in \mathcal {HI}$.
Hence every $\mathcal{HI}$-envelope is special.\\
(4) Since $\mathcal{CMI}$ is a covering class and  $\mathcal{HI}$
is an enveloping class, $(\mathcal{CMI}, \mathcal{HI})$ is
perfect.
\end{proof}

 Let $M$ be a module and $HIE(M)$ denote the Harmanci injective envelope of $M$.
By Theorem \ref{properties}(3), $HIE(M)/M$ belongs to
$\mathcal{CMI}$. We now investigate the case for injective
envelopes for Harmanci injective modules.

\begin{prop}\label{ess} Let $M$ be a Harmanci injective module and $N$ is an
essential extension of $M$. Then $N/M$ belongs to $\mathcal{CMI}$
if and only if $M=N$.
\end{prop}
\begin{proof} Let $N/M\in \mathcal{CMI}$ and consider the exact sequence $0\rightarrow M \rightarrow N \rightarrow N/M \rightarrow
0$. If we apply the functor Hom$_R(-,M)$ to the sequence, then we
obtain the exactness of Hom$_R(N,M)\rightarrow $
Hom$_R(M,M)\rightarrow $ Ext$^1_R(N/M, M)=0$ by the Harmanci
injectivity of $M$. It follows that $M$ is a direct summand of
$N$. So the essentiality of $M$ in $N$ yields $M=N$. The converse
is evident.
\end{proof}

\begin{thm}\label{inj ring} A module $M$ is injective if and only if $M$ is
Harmanci injective and $E(M)/M\in \mathcal{CMI}$.
\end{thm}
\begin{proof} The necessity is obvious. For the sufficiency, we have
$M=E(M)$ by Proposition \ref{ess}. This completes the proof.
\end{proof}

We now apply Theorem \ref{inj ring} to rings.

\begin{cor} The following are equivalent for a ring $R$.
\begin{enumerate}
    \item[(1)] $R$ is right self-injective.
    \item[(2)] $R$ is Harmanci injective as a right $R$-module, $E(R_R)$ is flat and $R$ is a pure submodule of $E(R_R)$.
    \item[(3)] $R$ is a Harmanci injective right $R$-module and  $E(R_R)/R\in \mathcal{CMI}$.
\end{enumerate}
\end{cor}
\begin{proof} (1) $\Rightarrow$ (2) Obvious. (3) $\Rightarrow$ (1)
By Theorem \ref{inj ring}.\\
(2) $\Rightarrow$ (3) Applying the functor $-\otimes_R E(_RR)$ to
the exact sequence $0\rightarrow R \rightarrow E(R_R)\rightarrow
E(R_R)/R\rightarrow 0$, we acquire the long exact sequence

$\cdots \rightarrow$ Tor$^R_1(E(R_R),E(_RR))\rightarrow $
Tor$^R_1(E(R_R)/R,E(_RR))\rightarrow R\otimes_R E(_RR)\rightarrow
E(R_R)\otimes_R E(_RR)\rightarrow E(R_R)/R\otimes_R
E(_RR)\rightarrow 0$. According to flatness of $E(R_R)$, we have
Tor$^R_1(E(R_R),E(_RR))=0$. Also by (2), the purity of $R$ in
$E(R_R)$ yields that Tor$^R_1(E(R_R)/R,E(_RR))=0$ as asserted.
\end{proof}

\begin{prop}  $\mathcal{CMI}$-covers of Harmanci injective
modules  are Harmanci injective.
\end{prop}
\begin{proof} Let $M$ be a Harmanci injective module and $f\colon C\rightarrow
M$ a $\mathcal{CMI}$-cover of $M$. According to Theorem
\ref{properties}(2), we have the exact sequence $0\rightarrow$
Ker$f\rightarrow C\rightarrow M\rightarrow 0$. Lemma
\ref{C-class}(6) implies that Ker$f$ is Harmanci injective. Hence
Lemma \ref{basic}(4) completes the proof.
\end{proof}

\begin{prop} Let $M$ be a right $R$-module and consider the following conditions.
\begin{enumerate}
    \item[(1)] $M$ is Harmanci injective.
    \item[(2)] For every exact sequence $0\rightarrow M\rightarrow B\overset{f}\rightarrow C \rightarrow
    0$ of right $R$-modules with $B\in \mathcal{CMI}$, $f$ is a $\mathcal{CMI}$-precover of $C$.
    \item[(3)] There exists a $\mathcal{CMI}$-precover $f\colon B\rightarrow
    C$ with $B$ Harmanci injective and Ker$f=M$.
\end{enumerate}
Then {\rm(3)} $\Rightarrow$ {\rm(1)} $\Rightarrow$ {\rm(2)}.
Furthermore if $\mathcal{CMI}$ is closed under homomorphic images,
then all of them are equivalent.
\end{prop}
\begin{proof} (3) $\Rightarrow$ (1) Let $N\in \mathcal{CMI}$ and consider the
exact sequence $0\rightarrow M\rightarrow B\overset{f}\rightarrow
C \rightarrow
    0$  where $f$ is a $\mathcal{CMI}$-precover and $B$ is Harmanci injective by (3).
Applying the functor Hom$_R(N,-)$, we behold the exactness of
Hom$_R(N,B)\rightarrow$ Hom$_R(N,C)\rightarrow $
Ext$^1_R(N,M)\rightarrow $ Ext$^1_R(N,B)\rightarrow \cdots$. The
homomorphism $f$ being a $\mathcal{CMI}$-precover and Harmanci
injectivity of $B$ imply Ext$^1_R(N,M)=0$.\\
(1) $\Rightarrow$ (2) Obvious by Proposition \ref{exact}.

Now assume that the class $\mathcal{CMI}$ is closed under
homomorphic images.\\
(2) $\Rightarrow$ (3) The natural projection $HIE(M)\rightarrow
HIE(M)/M$ is the required $\mathcal{CMI}$-precover due to Theorem
\ref{hom image}.
\end{proof}

We end this paper by investigating the unique mapping property of
$\mathcal{HI}$-envelopes.

\begin{thm}\label{unique} The following are equivalent for a ring $R$.
\begin{enumerate}
    \item[(1)] Every right $R$-module has an
    $\mathcal{HI}$-envelope with the unique mapping property.
    \item[(2)] For every exact sequence
    $0\rightarrow A \rightarrow B\rightarrow C$ of right
    $R$-modules, being $B, C\in \mathcal{HI}$ implies $A\in
    \mathcal{HI}$.
\end{enumerate}
\end{thm}
\begin{proof} (1) $\Rightarrow$ (2) Let $0\rightarrow A \overset{f}\rightarrow B \overset{g}\rightarrow C$ be  an exact sequence of right
$R$-modules with $B, C\in \mathcal{HI}$ and $h\colon A\rightarrow
D$ an $\mathcal{HI}$-envelope of $A$. Since $B$ is Harmanci
injective, there exists a unique homomorphism $\alpha\colon
D\rightarrow B$ such that $\alpha h= f$. Consider the following
diagram:
\begin{center}
$\xymatrix{ &D\ar@{.>}[dr]^\alpha \ar@{.>}@/_/[d]_\beta&&\\
0\ar[r]& A \ar[r]_f \ar[u]_h &B\ar[r]_g & C\\
&0\ar[u]&&}$
\end{center}
The equality $gf=0$ implies $g\alpha h=0=0h$, and so $g\alpha=0$
due to the unique mapping property of $h$. Then Im$\alpha\subseteq
$ Ker$g=$ Im$f$. By Factor Theorem, there exists a homomorphism
$\beta\colon D\rightarrow A$ with $f\beta=\alpha$. Hence $f\beta
h=\alpha h=f=f1_A$. Since $f$ is monic, $\beta h=1_A$, this yields
$D=$ Ker$\beta\oplus $ Im$h$. By Theorem \ref{properties}(3), it
is known that $h$ is monic, so $A\cong$ Im$h$. Thus this
isomorphism and
Lemma \ref{basic}(3) imply that $A$ is Harmanci injective.\\
(2) $\Rightarrow$ (1) Let $M$ be a right $R$-module and
$0\rightarrow M\overset{f}\rightarrow F$ an
$\mathcal{HI}$-envelope of $M$. Assume that for any $G\in
\mathcal{HI}$ and any homomorphism $g\colon M\rightarrow G$, there
exist $h_1, h_2 \colon F\rightarrow G$ such that $h_1f=h_2f=g$.
Since $(h_1-h_2)f=0$, Im$f\subseteq $ Ker$(h_1-h_2)$. Consider the
exact sequence $0\rightarrow $ Ker$(h_1-h_2)\rightarrow
F\overset{h_1-h_2}\rightarrow G$. By (2), Ker$(h_1-h_2)$ is
Harmanci injective. Then there exists a homomorphism $\beta\colon
F\rightarrow $ Ker$(h_1-h_2)\subseteq F$ satisfying $\beta f=if=f$
where $i\colon$ Im$f\rightarrow $ Ker$(h_1-h_2)$ is inclusion.
Since $f\colon M\rightarrow F$ is an $\mathcal{HI}$-envelope,
$\beta$ is an isomorphism. Hence $\beta(F)=$ Ker$(h_1-h_2)$.  This
yields $(h_1-h_2)\beta=0$. As $\beta$ has an inverse, $h_1-h_2=0$,
and so $h_1=h_2$. Therefore $f$ has the unique mapping property.
\end{proof}

By Theorem \ref{unique}, we immediately get the next result.

\begin{cor} Let $M$ be a Harmanci injective module and $N$ a submodule of $M$ with
$M/N$ Harmanci injective.  If the $\mathcal{HI}$-envelope of $N$
satisfies the unique mapping property, then $N$ is also Harmanci
injective.
\end{cor}


\begin{thebibliography}{99}

\bibitem{BBE} L. Bican, R. El Bashir and E. E. Enochs, {\it All
modules have flat covers}, Bull. London Math. Soc., 33(4)(2001),
385-390.

\bibitem{BBB} E. E. Bray, K. A. Byrd and R. L. Bernhardt, {\it The injective envelope of the upper triangular matrix ring},
 Amer. Math. Monthly,  78(1971), 883-886.



\bibitem{Ding} N. Q. Ding, {\it On envelopes with the unique
mapping property}, Comm. Algebra, 24(4)(1996), 1459-1470.

\bibitem{ES} P. C. Eklof and S. Shelah, {\it On Whitehead
modules}, J. Algebra, 142(2)(1991), 492-510.

\bibitem{Eno} E. E. Enochs, {\it Injective and flat covers,
envelopes and resolvents}, Israel J. Math., 39(3)(1981), 189-209.

\bibitem{Enochs} E. E. Enochs, {\it Flat covers and flat cotorsion modules}, Proc. Amer. Math. Soc., 92(2)(1984), 179-184.

\bibitem{Enochs-Jenda} E. E. Enochs and O. M. G. Jenda, Relative Homological
Algebra, De Gruyter Expositions in Mathematics, 30, Walter de
Gruyter \& Co., Berlin, 2000.

\bibitem{Harison} D. K. Harrison, {\it Infinite abelian groups and homological
methods}, Ann. of Math., 69(2)(1959), 366-391.

\bibitem{HJ} H. Holm and P. J{\o}rgensen, {\it Covers, precovers, and purity}, Illinois J. Math., 52(2)(2008), 691-703.



\bibitem{Lambek} J. Lambek, {\it A module is flat if and only if
its character module is injective}, Canad. Math. Bull.,
7(2)(1964), 237-243.

\bibitem{R} J. J. Rotman, An Introduction to Homological Algebra, Second edition, Universitext, Springer, New York, 2009.
\bibitem{Ro} L. H. Rowen, Ring Theory, Vol.1, Pure and Applied Mathematics, 127, Academic Press, Inc., Boston, MA, 1988.
\bibitem{T} J. Trlifaj, Covers, Envelopes, and Cotorsion Theories, Lecture
notes for the workshop ``Homological Methods in Module Theory",
Cortona, September 10-16, 2000.
\bibitem{Tr} J. Trlifaj, Infinite Dimensional Tilting Theory and
its Applications, Lecture notes for a series of talks at the
Nanjing University, June-July 2006.
\bibitem{X} J. Xu, Flat Covers of Modules, Lecture Notes in Mathematics, 1634, Springer-Verlag, Berlin, 1996.
\bibitem{Ya} H. Yan, {\it Matlis injective modules}, Bull. Korean
Math. Soc., 50(2)(2013),  459-467.

\end{thebibliography}
\end{document}